\documentclass[12pt]{amsart}

\usepackage{kantlipsum} 

\calclayout

\usepackage{amsmath, amssymb}
\usepackage{amscd}
\usepackage{verbatim}
\usepackage{tikz}
\usetikzlibrary{arrows,chains,matrix,positioning,scopes}

\makeatletter
\tikzset{join/.code=\tikzset{after node path={%
\ifx\tikzchainprevious\pgfutil@empty\else(\tikzchainprevious)%
edge[every join]#1(\tikzchaincurrent)\fi}}}
\makeatother

\tikzset{>=stealth',every on chain/.append style={join},
         every join/.style={->}}

\newtheorem{definition}{Definition}[section]
\newtheorem{theorem}[definition]{Theorem}
\newtheorem{lemma}[definition]{Lemma}
\newtheorem{corollary}[definition]{Corollary}
\newtheorem{proposition}[definition]{Proposition}
\theoremstyle{definition}
\newtheorem{remark}[definition]{Remark}

\newtheorem{example}[definition]{Example}

\newcommand{\noi}{\noindent}
\newcommand{\ra}{\rightarrow}

\begin{document}

\title[Analysis on Semihypergroups]{Analysis on Semihypergroups: Function Spaces, Homomorphisms and Ideals}

\author[C.~Bandyopadhyay]{Choiti Bandyopadhyay}

\address{Department of Mathematical and Statistical Sciences, University of Alberta, Canada T6G 2G1}
\email{choiti@ualberta.ca}

\thanks{This work is part of the author's PhD thesis at the University of Alberta}

\begin{abstract}
The main purpose of this article is to initiate a systematic study of Semihypergroups, first introduced by C. Dunkl \cite{DU}, I. Jewett \cite{JE} and R. Spector \cite{SP}  independently around 1972. We introduce and study several natural algebraic and analytic structures on semihypergroups, which are well-known in the case of topological groups and semigroups. In particular, we first study almost periodic and weakly almost periodic function spaces (basic properties, their relation to the compactness of the underlying space, introversion and Arens product on their duals among others). We then introduce homomorphisms and ideals, and thereby examine their behaviour (basic properties, structure of the kernel and relation of amenability to minimal ideals) in order to gain insight into the structure of a Semihypergroup itself. In the process, we further investigate where and why this theory deviates from the classical theory of semigroups.
\end{abstract}

\maketitle


\section{Introduction}

\quad The theory of topological groups and semigroups have been an important area of research in mathematics, specially from early 1960's. But in practice, one often comes across certain objects arising from groups (coset and double-coset spaces, orbit spaces \textit{etc.} for example), which although have a structure somewhat similar to groups, are not exactly groups.

\vspace{0.03in}

\quad Consider the double-coset space $GL_n(\mathbb{R}) // O(n)$ for example. Although it retains some group-like structures from $GL_n(\mathbb{R})$, it fails to be a group anymore. In fact, given any locally compact group $G$ and a compact subgroup $H$ of it, the double-coset space $G//H$ fails to be a group, hence forcing us to treat it individually whenever we come across these kinds of objects in research. We experience similar difficulties while investigating the orbit space of a continuous affine action of a compact group on a locally compact group (more discussion on this will follow later) as well. 

\vspace{0.03in}

\quad Around 1972, three mathematicians, C. Dunkl \cite{DU}, I. Jewett \cite{JE} and R. Spector \cite{SP} independently came up with a unified theory involving convolution of measures, in pursuit of moulding these kinds of objects into an unified class of objects called Hypergroups.

\vspace{0.03in}

\quad A hypergroup may essentially be seen as a generalized form of a locally compact group. Only here, the product of two points is a probability measure, rather than a single point; and as expected, it contains an identity and a certain `involution', which is an interesting generalization for inverses in the group theory \cite{JE}.

\vspace{0.03in}

\quad It turns out that all the double-coset spaces and the orbit spaces where the individual mappings $g\mapsto h.g: G\ra G$ are automorphisms, can be given a natural hypergroup-structure; and hence their properties can directly be studied and derived from the study of the unified theory of hypergroups. Since the initiation of systematic studies on it, there has been rapid progress in the research on several areas involving hypergroups, including the area of commutative and weighted hypergroups, amenability and several function spaces and algebras on it.

\vspace{0.03in}

\quad The next question is, what happens when the orbit spaces of continuous affine actions do not permit all the individual maps $g\mapsto h.g : G \ra G$ to be automorphisms. For example, what happens when instead of double-coset spaces, we encounter homogeneous spaces or simply left-coset spaces like $GL_n(\mathbb{R})/O_n(\mathbb{R})$, $GL_n(\mathbb{R})/SO_n(\mathbb{R})$, $S_4/D_8$, $S_4/S_3$ or for that matter, any left-coset space $G/H$, $G$ being any locally compact group and $H$ being a compact non-normal subgroup of it. It is apparent that these spaces, although retaining some group-like structures, are not topological groups. More importantly, none of these spaces can even be given a hypergroup-structure (reasons explained in the third section). Hence these kinds of objects that frequently appear while studying the classical theory of topological groups, fall out of the parent category and can not quite be studied or analysed with the existing general theory for topological groups and hypergroups.

\vspace{0.03in}   

\quad The solution to this problem can be attained through a detailed study of a more general category of objects involving convolution algebra of measures, called Semihypergroups. They were introduced at the same time as hypergroups \cite{DU, JE, SP} and in a sense, served as building blocks for hypergroup-axioms. A semihypergroup, as one would expect, can also be seen as a generalized form of a locally compact semigroup, where the product of two points is a certain probability measure, rather than a single point. But unlike hypergroups, it does not need to have an identity or involution. Hence in a nutshell, a semihypergroup is essentially a locally compact topological space where the measure space is equipped with a certain convolution product, turning it into an associative algebra; whereas in the case of a hypergroup, the measure algebra is also equipped with an identity and an involution.

\vspace{0.03in} 

\quad It can be seen (details in the third section) that all the orbit spaces and coset spaces discussed above can naturally be given a semihypergroup structure, though not a hypergroup structure. The fact that semihypergroups contain more generalized objects arising from different fields of research than classical group and hypergroup theory and yet sustains enough structure to allow an independent theory to develop, makes it an intriguing and useful area of study with essentially a broader area of applications. 

\vspace{0.03in}

\quad Unlike hypergroups, there is a severe lack of any extensive prior research on Semihypergroups since its inception. The significant examples it contains, opens up a number of new intriguing paths of research on semihypergroups. The lack of an algebraic structure on the underlying space poses a serious challenge in extending results from semigroups to semihypergroups. Also unlike hypergroups, the  fact that a semihypergroup structure lacks the existence of a Haar measure or an involution in its measure algebra, creates a serious obstacle to the extension of most of the important group and hypergroup theories and ideas naturally to semihypergroups.

\vspace{0.03in}
 
\quad Hence the main motivation of this paper is to initiate and develop a systematic study of semihypergroups. We will base our work on semihypergroups on Jewett's  definition of `semiconvos' \cite{JE}. In this particular paper, we will restrict ourselves to the discussion of some particular functions spaces, ideals and homomorphisms of a semihypergroup. The rest of the paper is designed as below.

\vspace{0.03in}

\quad In the next, \textit{i.e}, second section of this article, we recall some preliminary definitions and notations given by Jewett in \cite{JE}, and introduce some new definitions required for further work. In the third section, we list some important useful examples of semihypergroups and hypergroups.

\vspace{0.03in}

\quad In our fourth section we discuss two of the most important function spaces on semihypergroups, namely the spaces of almost periodic and weakly almost periodic functions. Here we explore the relation between the left and right counterparts of these function spaces (Theorem \ref{wapl=r}, Corollary \ref{apl=r}) and show how the compactness of the underlying space influences amenability and the structure of these function spaces, depending on the kind of continuity we have on the measure algebra of the underlying space (Proposition \ref{sh1}, Theorems \ref{sh2}, \ref{amen1}). We then introduce introversion on the function spaces and conclude the section with examining Arens regularity on the duals of these function spaces (Corollary \ref{introcor}, Theorem \ref{intro00}).

\vspace{0.03in}

\quad In the fifth section, we deal with the inner structure of a semihypergroup itself. We introduce the concept of an ideal in (semitopological) semihypergroups and explore some of its basic properties as well as its relation to a more general form of homomorphism between semihypergroups. Furthermore, we investigate the structure of the kernel of a compact (semitopological) semihypergroup (Theorem \ref{ithm1}), giving us insight into the algebraic structure of a semihypergroup. The discussion also allows us insight into where exactly this theory starts deviating from the classical theory of topological semigroups (for example, Remark \ref{idealdevi}). Finally, we conclude the section with exploring the connection between minimal left ideals and amenability on a compact semihypergroup (Theorem \ref{ithm3}). 

\vspace{0.03in}

\quad Finally, this article is the initiation of a systematic study of this class of objects with intriguing scope for application in the theory of coset and orbit spaces of locally compact groups, homogeneous spaces and Lie groups. We conclude with some potential problems and areas which we intend to work on and explore further in near future, in addition to introducing and investigating similar structures as in \cite{FO1}, \cite{FO2}, \cite{FO3}, \cite{GH1}, \cite{GH2}, \cite{GH3}, \cite{LA2}, \cite{LDS}, \cite{NA}, \cite{WI1}, \cite{WI2} for the case of semihypergroups and explore further where and why the theory deviates from the classical theory of topological semigroups and groups.
\section{Notations and Definitions}
\noi Here we first list some of the preliminary set of basic notations that we will use throughout, followed by a brief introduction to the tools and concepts needed for the formal definition of a semihypergroup \cite{JE}. All the topologies throughout this article are assumed to be Hausdorff, unless otherwise specified.

\vspace{0.03in}

\noi For any locally compact Hausdorff topological space $X$, we denote by $M(X)$ the space of all regular complex Borel measures on $X$, where $M^+(X)$ and $P_c(X)$ denote the subsets of $M(X)$ consisting of all finite non-negative regular Borel measures on $X$ and all probability measures with compact support on $X$ respectively. Moreover, $B(X), \mathcal{B}(X), C(X)$ and $C_c^+(X)$ denote the function spaces of all bounded functions, Borel measurable functions, bounded continuous functions and non-negative compactly supported continuous functions on $X$ respectively.

\vspace{0.03in}

\noi Next, we introduce two very important topologies on the positive measure space and the space of compact subsets for any locally compact topological space $X$. Unless mentioned otherwise, we will always assume these two topologies on the respective spaces.

\vspace{0.03in}

The \textbf{cone topology} on $M^+(X)$ is defined as the weakest topology on $M^+(X)$ for which the maps $ \mu \mapsto \int_X f \ d\mu$ is continuous for any $f \in C_c^+(X)\cup \{1_X\}$ where $1_X$ denotes the characteristic function of $X$. Note that if $X$ is compact then it follows immediately from the Riesz representation theorem that the cone topology coincides with the weak*-topology on $M^+(X)$ in this case.

\vspace{0.03in}

We denote by $\mathfrak{C}(X)$ the set of all compact subsets of $X$. The \textbf{Michael topology} on $\mathfrak{C}(X)$ is defined to be the topology generated by the sub-basis $\{\mathcal{C}_U(V) : U,V \textrm{ are open sets in } X\}$, where
$$\mathcal{C}_U(V) = \{ C\in \mathfrak{C}(X) : C\cap U \neq \O, C\subset V\}.$$

\noi Note that $\mathfrak{C}(X)$ actually becomes a locally compact Hausdorff space with respect to this natural topology. Moreover if $X$ is compact then $\mathfrak{C}(X)$ is also compact. 

\vspace{0.03in}

For any locally compact Hausdorff space $X$ and any element $x\in X$, we denote by $p_x$ the point-mass measure or the Dirac measure at the point $\{x\}$.

\vspace{0.03in}

For any three locally compact Hausdorff spaces $X, Y, Z$, a bilinear map $\Psi : M(X) \times M(Y) \rightarrow M(Z)$ is called \textbf{positive continuous} if the following holds :
\begin{enumerate}
\item $\Psi(\mu, \nu) \in M^+(Z)$ whenever $\mu \in M^+(X), \nu \in M^+(Y)$.
\item The map $\Psi |_{M^+(X) \times M^+(Y)}$ is continuous.
\end{enumerate}

\vspace{0.03in}

\noi Now we are ready to state the formal definition for a semihypergroup. Note that we follow Jewett's notion \cite{JE} in terms of the definitions and notations, in most cases.

\begin{definition}\label{shyper}\textbf{(Semihypergroup)} A pair $(K,*)$ is called a (topological) semihypergroup if it satisfies the following properties:


\begin{description} 
\item[(A1)] $K$ is a locally compact Hausdorff space and $*$ defines a binary operation on $M(K)$ such that $(M(K), *)$ becomes an associative algebra.

\item[(A2)] The bilinear mapping $* : M(K) \times M(K) \rightarrow M(K)$ is positive continuous.

\item[(A3)] For any $x, y \in K$ the measure $p_x * p_y$ is a probability measure with compact support.

\item[(A4)] The map $(x, y) \mapsto \mbox{ supp}(p_x * p_y)$ from $K\times K$ into $\mathfrak{C}(K)$ is continuous.
\end{description}
\end{definition}

\noi A semihypergroup $(K, *)$ is called a hypergroup if $K$ admits an identity element $e \in K$ and a certain involution map (analogous to inverses in groups). Note that for any $A,B \subset K$ the convolution of subsets is defined as the following:
$$A*B := \cup_{x\in A, y\in B} \ supp(p_x*p_y) $$

\noi  Now we conclude this section by introducing the concept of left (or right) topological and semitopological semihypergroups, similar to the concepts of left (or right) topological and semitopological semigroups.

\begin{definition}
A pair $(K, *)$ is called a left topological semihypergroup if it satisfies all the conditions of Definition \ref{shyper} with property \textbf{(A$2$)} replaced by the following:
\begin{description}
\item[(A2$'$)] The map $(\mu, \nu) \mapsto \mu*\nu$ is positive and for each $\omega \in M^+(K)$ the map $L_\omega:M^+(K) \rightarrow M^+(K)$ given by $L_\omega(\mu)= \omega*\mu$ is continuous.
\end{description}
\end{definition}

\begin{definition}
A pair $(K, *)$ is called a right topological semihypergroup if it satisfies all the conditions of Definition \ref{shyper} with property \textbf{(A$2$)} replaced by the following:
\begin{description}
\item[(A2$''$)] The map $(\mu, \nu) \mapsto \mu*\nu$ is positive and for each $\omega \in M^+(K)$ the map $R_\omega:M^+(K) \rightarrow M^+(K)$ given by $R_\omega(\mu)= \mu*\omega$ is continuous.
\end{description}
\end{definition}

\begin{definition}
A pair $(K, *)$ is called a \textbf{semitopological semihypergroup} if it is both left and right topological semihypergroup.
\end{definition}


\section{Examples}

\noi In this section, we list some well known examples \cite{JE,ZE} of semihypergroups and hypergroups. Thus we explore how the shortcomings explained in the introduction are overcome by the category of semihypergroups and hypergroups, as well as why most of the structures discussed there, although they attain semihypergroup structures, fail to be hypergroups.

\begin{example} \label{extr}
If $(S, \cdot)$ is a locally compact topological semigroup, then $(S, *)$ is a semihypergroup where $p_x*p_y = p_{x.y}$ for any $x, y \in S$.

\vspace{0.05cm}

\noi Similarly, if $(G, \cdot)$ is a locally compact topological group, then $(G, *)$ is a hypergroup with the same bilinear operation $*$, identity element $e$ where $e$ is the identity of $G$ and the involution on $G$ defined as $x \mapsto x^{-1}$.
\end{example}

\begin{example} \label{ex2}
Take $T = \{e, a, b\}$ and equip it with the discrete topology. Define
\begin{eqnarray*}
p_e*p_a &=& p_a*p_e \ = \ p_a\\
p_e*p_b &=& p_b*p_e \ = \ p_b\\
p_a*p_b &=& p_b*p_a \ = \ z_1p_a + z_2p_b\\
p_a*p_a &=& x_1p_e + x_2p_a + x_3p_b \\
p_b*p_b &=& y_1p_e + y_2p_a + y_3p_b
\end{eqnarray*}

\noi where $x_i, y_i, z_i \in \mathbb{R}$ such that $x_1+x_2+x_3 = y_1+y_2+y_3 = z_1+z_2 = 1$ and $y_1x_3 = z_1x_1$. Then $(T, *)$ is a commutative hypergroup with identity $e$ and the identity function on $T$ taken as involution.
\end{example}

\begin{example} \label{ex3}
Let $G$ be a locally compact topological group and $H$ be a compact subgroup of $G$. Also, let $\mu$ be the normalized Haar measure of $H$. Consider the left quotient space $S := G/H = \{xH : x \in G\}$ and equip it with the quotient topology. For any $x, y \in G$, define
$$p_{xH} * p_{yH} = \int_H p_{(xty)H} \ d\mu(t) .$$
\noi Then $(S, *)$ is a semihypergroup.

\noi For instance, take $G$ to be the symmetric group $S_4$ and take $H$ to be the dihedral group $D_8$. We know that $D_8$ is not a normal subgroup of $S_4$. Consider the left coset space 
$$G/H = \{H, s_1H, s_2H\}$$
where $s_1=(1 2 4)$ and $s_2=(1 4 2)$. Then the above formulation gives us that
$$p_{xH} * p_{yH} = \frac{1}{8} \sum_{h\in H} p_{(xty)H} $$
\noi Hence a direct computation gives us that the left coset space $S_4/D_8$ is a discrete semihypergroup where the convolution is given by the following table:
\begin{center}
  \begin{tabular}{ c | c  c  c }
    $*$ & $p_H$ & $p_{s_1H}$ & $p_{s_2H}$ \\
    \hline
    $p_{H}$ & $p_H$ & $\frac{1}{2}(p_{s_1H}+p_{s_2H})$ & $\frac{1}{2}(p_{s_1H}+p_{s_2H})$ \\ 
    $p_{s_1H}$ & $p_{s_1H}$ & $\frac{1}{2}(p_{H}+p_{s_2H})$ & $\frac{1}{2}(p_{H}+p_{s_1H})$ \\
    $p_{s_2H}$ & $p_{s_2H}$ & $\frac{1}{2}(p_{H}+p_{s_1H})$ & $\frac{1}{2}(p_{H}+p_{s_2H})$ \\
  \end{tabular}
\end{center}

\vspace{0.09cm}

Note that here the element $x=H$ fails to be a right identity, and hence $S_4/D_8$ does not attain a hypergroup structure. Thus in general,  for a left coset space $G/H$ where $H$ is not normal, $H$ fails to serve as a right identity.

\end{example}

The next example \cite{JE} of double-coset spaces  overcomes this barrier and becomes a hypergroup using a similar convolution product.


\begin{example} \label{ex5}
Consider $G$ and $H$ as in the previous example and equip the space of double cosets \\$K := G/ /H = \{HxH : x \in G\}$ with the usual quotient topology. For any $x, y \in G$, define
$$p_{HxH} * p_{HyH} = \int_H p_{H(xty)H} \ d\mu(t) .$$
\noi Then $(K, *)$ is a hypergroup with the identity element $e = H$ and involution function $HxH \mapsto Hx^{-1}H$.
\end{example}

Finally, the next example of orbit spaces includes \cite{JE} both Examples \ref{ex3}, \ref{ex5}. Recall that a continuous action of a topological group $T$ on a Hausdorff space $X$ is a continuous map $(t, x) \mapsto x^t : T\times X \ra X$ such that $x^e=x$ and $x^{st} = (x^s)^t$ for each $x\in X$, $s, t\in T$.

\begin{example} \label{ex4}
Let $G$ be a locally compact topological group and $H$ be any compact group. Recall that a map $\phi: G\ra G$ is called affine if there exists a scaler $\alpha$ and an automorphism $\Psi$ of $G$ such that $\phi(x) = \alpha \Psi(x)$ for each $x\in G$.

\vspace{0.05cm}

\noi A continuous action $\pi$ of $H$ on $G$ given by $\pi(h, g) = g^h$ for each $g\in G, h\in H$ is called a continuous affine action if the map $g\mapsto g^h: G\ra G$ is affine for each $h\in H$. For any continuous affine action $\pi$ of $H$ on $G$, consider the orbit space
$$ \mathcal{O} := \{x^H : x \in G\},$$
where $x^H = \mathcal{O}(x) = \{\pi(h, x): h\in H\}$. 

\vspace{0.05cm}

\noi Let $\sigma$ be the normalized Haar measure of $H$. Consider $\mathcal{O}$ with the quotient topology and the following convolution
$$ p_{x^H} * p_{y^H} := \int_H \int_H p_{(\pi(s, x)\pi(t, y))^H} \ d\sigma(s) d\sigma(t) .$$
\noi Then $(\mathcal{O}, *)$ becomes a semihypergroup.
\end{example}

\section{Almost Periodic And Weakly Almost Periodic Functions}

\noi In this section, we start with testing some basic properties of the spaces of almost periodic and weakly almost periodic functions for a (semitopological) semihypergroup, and investigate the relation between these spaces and compactness of the underlying space. Finally, we conclude the section with examining Arens regularity \cite{AR} on the dual of the space of almost periodic functions.

\vspace{0.05cm}

\noi Recall \cite{JE} that for any measurable function $f$ on a (semitopological) semihypergroup $K$ and each $x, y \in K$ we define the left and right translates of $f$ (denoted as $L_xf$ and $R_xf$ respectively) as follows.
$$ L_xf(y) = R_yf(x) = f(x*y) = \int_K f \ d(p_x*p_y)\ $$

For any function $f\in C(K)$ we define the right orbit $\mathcal{O}_r(f)$ of $f$ as
$$\mathcal{O}_r(f) := \{R_xf : x\in K\} .$$

A function $f\in C(K)$ is called right almost periodic if $\mathcal{O}_r(f)$ is relatively compact in $C(K)$ with respect to the norm topology. Similarly, a function $f \in C(K)$ is called right weakly almost periodic if $\mathcal{O}_r(f)$ is relatively compact in $C(K)$ with respect to the weak topology on $C(K)$.

We denote these two classes of functions as:
\begin{eqnarray*}
AP_r(K) &:=&  \mbox{Space of all right almost periodic functions on } K.\\
WAP_r(K) &:=& \mbox{Space of all right weakly almost periodic functions on } K.
\end{eqnarray*}

\noi Similarly we can define the left orbit $\mathcal{O}_l(f)$ of a function $f\in C(K)$ and hence define the set of all left almost periodic functions $AP_l(K)$ and left weakly almost periodic functions $WAP_l(K)$. In most instances, the results we prove here for the right-case also hold true for the left-case.

\vspace{0.03in}

A function $f\in C(K)$ is called left(right) uniformly continuous if the map $x\mapsto L_xf$ ($x\mapsto R_xf$) from $K$ to $C(K)$ is continuous. A function $f$ is called uniformly continuous if it is both left and right uniformly continuous. The space consisting of  such functions is denoted by $UC(K)$.

\vspace{0.03in}

Let $K$ be a (semitopological) semihypergroup with identity and $\mathcal{F}$ a linear subspace of $C(K)$ containing constant functions. A function $m\in \mathcal{F}^*$ is called a mean of $\mathcal{F}$ if we have that $||m|| = 1 = m(1)$. If $\mathcal{F}$ is a translation-invariant linear subspace of $C(K)$ containing constant functions, a mean $m$ of $\mathcal{F}$ is called a left invariant mean (LIM) if $m(L_xf) = m(f)$ for any $x\in K$, $f\in \mathcal{F}$. We can similarly define a right invariant mean (RIM) for $\mathcal{F}$. We say that a (semitopological) semihypergroup $K$ is left (resp. right) amenable if $C(K)$ admits a LIM (resp. RIM).

\vspace{0.03in}

\noi Now we examine the relation between the spaces of left and right (weakly) almost periodic functions. But before we  prove the first result, let us first recall an important result on the existence of a vector integral on the function space of a general locally compact Hausdorff space, proved in \cite{RU}.

\begin{theorem} \label{RU1}
Suppose that $(X,\tau)$ is a Hausdorff topological vector space where $X^*$ separates points, and $\lambda \in P_c(Y)$ where $Y$ is a locally compact Hausdorff space. If a function $F:Y\ra X$ is continuous and $\overline{co}(F(Y))$ is compact in $X$ then the vector integral $\omega:= \int_Y F\ d\lambda$ exists and $\omega \in \overline{co}(F(Y))$.
\end{theorem}

\noi We divide the proof of our next theorem in a series of key steps for convenience, due to the length of the proof.


\begin{theorem} \label{wapl=r}
Let $K$ be a semihypergroup. Then 
$$WAP_r(K) \cap UC(K) = WAP_l(K) \cap UC(K). $$
\end{theorem}
\begin{proof}
Pick any $f\in WAP_l(K) \cap UC(K)$. We need to show that $f\in WAP_r(K)$, \textit{i.e,} $\mathcal{O}_r(f)$ is relatively weakly compact in $C(K)$. We will show this in five steps.


\noi \textbf{Step I:} \hspace{0.01in} Embed $K$ into $P_c(K)$ through the homomorphism $x\mapsto p_x$ \cite{JE}. Now define a map $\Phi:C(K) \ra C(P_c(K))$ by $\Phi(f) = \tilde{f}$ for any $f\in C(K)$ where
$$\tilde{f}(\mu) := \int f \ d\mu \ \ \ \forall \ \mu \in P_c(K).$$
\noi We know that $\Phi$ is an isometry \cite{JE}. Thus $\Phi$ is continuous in the norm topologies of $C(K)$ and $C(P_c(K))$ and hence $\Phi$ is continuous in the weak topologies of $C(K)$ and $C(P_c(K))$.

\vspace{0.05cm}

\noi \textbf{Step II:} \hspace{0.01in} For any $\mu \in P_c(K)$ define a function $\hat{L}_\mu f$ on $K$ by
$$\hat{L}_\mu f(x) := \int_K f(y*x) \ d\mu(y) .$$
\noi Since $f\in UC(K)$ the map $x\mapsto R_xf$ is continuous, and $\hat{L}_\mu f(x) = \int_K R_xf \ d\mu$ by the above construction. Hence $\hat{L}_\mu f\in C(K)$ for any $\mu \in P_c(K)$. Now define a subset in $C(K)$ as the following.
$$\tilde{\mathcal{O}}_l(f) := \{ \hat{L}_\mu f : \mu \in P_c(K)\}.$$

\noi Now consider the left-translation map $\psi: K\ra C(K)$  given by $\psi(x) := L_xf$. Since $f\in UC(K)$ $\psi$ is continuous on $K$. Also since $f\in WAP_l(K)$ the left orbit of $f$ namely $\psi(K)$ is relatively weakly compact in $C(K)$.

\noi Hence by the Krein-Smulian Theorem we have that $\overline{co}(\psi(K))$ is weakly compact in $C(K)$, \textit{i.e,} $\overline{co}(\mathcal{O}_l(K))$ is weakly compact in $C(K)$.

\vspace{0.05cm}

\noi \textbf{Step III:} \hspace{0.01in} Now can use Theorem \ref{RU1} by letting $(X, \tau) = (C(K), \mbox{ weak topology})$, $Y=K$, $F=\psi$ and $\lambda =\mu$ for any $\mu\in P_c(K)$. Thus we see that $\omega_0 := \int_K\psi \ d\mu \in \overline{co}(\psi(K))$. Now for any $x\in K$ we have that
\begin{eqnarray*}
\omega_0(x) &=& \int_K (\psi(y)) (x) \ d\mu(y) \\
&=& \int_K L_yf(x) \ d\mu(y) \\
&=& \int_K f(y*x) \ d\mu(y) = \hat{L}_\mu f(x).
\end{eqnarray*}
\noi Since $\psi(K) = \mathcal{O}_l(f)$ by construction, we see that $\hat{L}_\mu f \in \overline{co} (\mathcal{O}_l(K))$ for any $\mu \in P_c(K)$. Hence $\tilde{\mathcal{O}}_l(f) \subset \overline{co} (\mathcal{O}_l(K))$. But from Step II we know that $ \overline{co} (\mathcal{O}_l(K))$ is weakly compact in $C(K)$. Hence $\tilde{\mathcal{O}}_l(f)$ is relatively weakly compact in $C(K)$.

\vspace{0.05cm}

\noi \textbf{Step IV:} \hspace{0.01in} Note that $(\hat{L}_\mu f)\tilde{} = L_\mu \tilde{f}$ for any $\mu \in P_c(K)$. This is true since for any $\nu \in P_c(K)$ we have that
\begin{eqnarray*}
(\hat{L}_\mu f)\tilde{}(\nu) &=& \int_K \hat{L}_\mu f (x) \ d\nu(x)\\
&=& \int_K\int_K f(y*x) \ d\mu(y) \ d\nu(x)\\
&=& \int_K f \ d(\mu*\nu)\\
&=& \tilde{f}(\mu*\nu) = L_\mu \tilde{f}(\nu).
\end{eqnarray*}
\noi where the third equality follows from \cite[Theorem 3.1E]{JE}. Thus we see that
 $$(\tilde{\mathcal{O}}_l(f))\tilde{} = \{\tilde{f} : f\in \tilde{\mathcal{O}}_l(f) \subset C(K)\} = \mathcal{O}_l(\tilde{f}).$$
\noi From Step III we know that $\tilde{\mathcal{O}}_l(f)$ is relatively weakly compact in $C(K)$, and from Step I we see that the map $f\mapsto \tilde{f}$ from $C(K)$ to $C(P_c(K))$ is continuous when both spaces are equipped with weak topology. Hence the set $(\tilde{\mathcal{O}}_l(f))\tilde{}$ is also relatively weakly compact in $C(P_c(K))$. Thus $\mathcal{O}_l(\tilde{f})$ is relatively weakly compact in $C(P_c(K))$, \textit{i.e,} $\tilde{f} \in WAP_l(P_c(K))$.

\vspace{0.05cm}

\noi \textbf{Step V:} \hspace{0.01in} We know that $(P_c(K), *)$ is a topological semigroup. Hence $WAP_l(P_c(K)) = WAP_r(P_c(K))$. Thus we have that $\tilde{f} \in WAP_r(P_c(K))$, \textit{i.e,} $\mathcal{O}_r(\tilde{f})$ is relatively weakly compact in $C(P_c(K))$.

\vspace{0.05cm}

\noi Hence the set $N:= \{R_{p_x}\tilde{f} : x\in K\} \subset \mathcal{O}_r(\tilde{f})$ is also relatively weakly compact in $C(P_c(K))$.

\vspace{0.05cm}

\noi Now consider the map $\phi: C(P_c(K)) \ra C(K)$ given by $h\mapsto \check{h}$ where $\check{h}(x) := h(p_x)$ for any $x\in K$. The fact that $\check{h}$ is continuous on $K$ follows directly from the fact that the map $x \mapsto p_x$ is continuous. Also, $\phi$ is continuous since $||\check{h}|| = \sup_{x \in K} |h(p_x)| \leq ||h||$ for any $h\in C(P_c(K))$.

\vspace{0.05in}

\noi Note that for each $h\in C(K)$, $x\in K$, $\phi(R_{p_x}\tilde{h}) = R_xh$ since for any $y\in K$ we have
\begin{eqnarray*}
\phi(R_{p_x}\tilde{h})(y) &=& (R_{p_x}\tilde{h})\check{} (y)\\
&=& \tilde{h} (p_y * p_x)
= h(y*x) = R_x h(y).
\end{eqnarray*}
\noi Thus in particular, we have that $\phi(N) = \mathcal{O}_r(f)$. Since $\phi$ is continuous and $N$ is relatively weakly compact in $C(P_c(K))$, we finally have that $\mathcal{O}_r(f)$ is relatively weakly compact in $C(K)$, \textit{i.e,} $f\in WAP_r(K)$ as required.

\noi Similarly for any $f\in WAP_r(K) \cap UC(K)$ we can show that $f\in WAP_l(K)$. Hence the proof is complete.
\end{proof}

\begin{remark}
In the proof of Theorem \ref{wapl=r}, the ideas involved in steps I-III are similar to \cite[Theorem 2.4]{WO} as no specific involution properties are required. However, the above proof of those steps contains several important details, not found in the aforementioned text. 
\end{remark}

\begin{corollary} \label{apl=r}
Let $K$ be a semihypergroup. Then $AP_l(K) = AP_r(K)$.
\end{corollary}
\begin{proof}
This can be proved in exactly similar manner as in Theorem \ref{wapl=r}. The only modification needed is that, we need to consider the norm topologies on $C(K)$ and $C(P_c(K))$ as opposed to the fact that we had to consider weak topologies on these two spaces in the proof of Theorem \ref{wapl=r}.
\end{proof}

\noi The following facts, regarding the relation between the function spaces in question, are proved for hypergroups in \cite{WO}. The same proof works for semihypergroups as well.

\begin{proposition}
Let $K$ be a semihyergroup. Then the following statements hold true.
\begin{enumerate}
\item $WAP_l(K), WAP_r(K), AP_l(K), AP_r(K)$ are norm-closed, conjugate-closed subsets of $C(K)$ containing the constant functions.
\item $AP_l \subset UC(K)$ and $AP_r(K) \subset UC(K)$.
\item $AP_l(K)\subset WAP_l(K)$ and $AP_r(K)\subset WAP_r(K)$.
\end{enumerate}
\end{proposition}

\noi Next we examine the translation-invariance properties of the spaces of (weakly) almost periodic functions on a (semitopological) semihypergroup. Recall that a function space $\mathcal{F}$ on a (semitopological) semihypergroup $K$  is called left (or right) translation-invariant if we have that $L_xf \in \mathcal{F}$ (or $R_xf \in \mathcal{F}$) for any $f\in \mathcal{F}$ and $x\in K$. Also, $\mathcal{F}$ is simply said to be translation-invariant if it is both left and right translation-invariant.

\begin{proposition} \label{aptr}
Let $K$ be any semitopological semihypergroup. Then $AP(K)$ is translation-invariant.
\end{proposition}
\begin{proof}
Pick any $f\in AP(K)$, $x_0 \in K$. Now consider the map $\Phi: \ C(K) \ra  C(K)$ given by $g \mapsto L_{x_0}g$.
Note that for any $x, y, t \in K$ we have that 
$$R_xL_yf(t) = L_yf(t*x) = f(y*t*x) = R_xf(y*t) = L_yR_xf(t) .$$ 
\noi Hence in turn we have
\begin{eqnarray*}
\mathcal{O}_r (L_{x_0}f) &=& \{R_xL_{x_0}f : x\in K\}\\
&=& \{L_{x_0}R_xf  : x \in K\} = \Phi(\mathcal{O}_r(f)).
\end{eqnarray*}
\noi Now for any two functions $g, h \in C(K)$ we have that
\begin{eqnarray*}
||\Phi(g) - \Phi(h)|| &=& \sup_{y\in K} |L_{x_0}(g-h) (y)|\\
&\leq & ||g-h|| \ \sup_{y\in K} |(p_{x_0}*p_y)(K)| = ||g-h|| \ .
\end{eqnarray*}
\noi Thus we see that $\Phi$ is continuous in the norm topology on $C(K)$. Hence $\Phi(\overline{\mathcal{O}_r(f)})$ is compact since $f\in AP_r(K)$. Also as noted before, since $\mathcal{O}_r(L_{x_0}f) = \Phi(\mathcal{O}_r(f)) \subset \Phi(\overline{\mathcal{O}_r(f)})$ we have that $\mathcal{O}_r(L_{x_0}f)$ is relatively compact in $C(K)$ with respect to the norm-topology and hence $L_{x_0}f \in AP_r(K) = AP(K)$. A similar argument, using the map $\Psi: \ C(K) \ra C(K)$ given by $g \mapsto R_{x_0}g$, shows that $\mathcal{O}_l (R_{x_0}f)$ is relatively compact in $C(K)$ so that $R_{x_0}f \in AP_l(K) = AP(K)$ as required.
\end{proof}

\begin{proposition}
Let $K$ be any semitopological semihypergroup. Then 
\begin{enumerate}
\item $WAP_r(K)$ is left translation-invariant.
\item $WAP_l(K)$ is right translation-invariant.
\end{enumerate}
\end{proposition}
\begin{proof}
Pick any $f\in WAP_r(K)$, $x_0 \in K$. As in the proof of the above theorem, consider the map $\Phi: \ C(K) \ra  C(K)$ given by $g \mapsto L_{x_0}g$.
Note that $\overline{\mathcal{O}_r(f)}^w$ is compact in $C(K)$. But the weak topology on $C(K)$ is stronger than the topology of pointwise-convergence where the latter topology is Hausdorff. Hence these two topologies coincide on any compact subset on $C(K)$, in particular on $\overline{\mathcal{O}_r(f)}^w$.

\noi Now let $g \in \overline{\mathcal{O}_r(f)}^w$ and $\{g_\alpha\}$ be a net in $\overline{\mathcal{O}_r(f)}^w$ such that $g_\alpha\longrightarrow^{\hspace{-0.16in} w} \ g$ on $K$. Then we have
\begin{eqnarray*}
&& \int_K g_\alpha \ d(p_{x_0}*p_y) \ra  \int_K g \ d(p_{x_0}*p_y) \ \  \mbox{for each } y \in K.\\
&\Rightarrow & L_{x_0}g_\alpha \ra  L_{x_0}g \ \  \mbox{ pointwise on } K.\\
&\Rightarrow & L_{x_0}g_\alpha \longrightarrow^{\hspace{-0.16in} w} \ L_{x_0}g \ \  \mbox{ on } C(K).
\end{eqnarray*}
\noi Thus $\Phi$ is weak-weak continuous on $\overline{\mathcal{O}_r(f)}^w$ and hence $\Phi(\overline{\mathcal{O}_r(f)}^w)$ is weakly compact in $C(K)$. Now the fact that $L_{x_0}f \in WAP_r(K)$ follows from the fact that 
\begin{eqnarray*}
\mathcal{O}_r (L_{x_0}f) &=& \{R_xL_{x_0}f : x\in K\}\\
&=& \{L_{x_0}R_xf  : x \in K\} = \Phi(\mathcal{O}_r(f)) \subseteq  \Phi(\overline{\mathcal{O}_r(f)}^w).
\end{eqnarray*}
\noi In a similar manner we can also see that $WAP_l(K)$ is right translation-invariant.
\end{proof}

\noi Next we examine the behavior of the (weakly) almost periodic function spaces when the underlying (semitopological) semihypergroup is compact. Before we proceed further, let us first quickly recall a widely-used result proved by A. Grothendieck \cite{GRO}.

\begin{theorem}[Grothendieck]\label{grothm}
Let $X$ be a compact Hausdorff space. Then a bounded set in $C(X)$ is weakly compact if and only if it is compact in the topology of pointwise convergence.
\end{theorem}

\noi The next Proposition also follows easily from \cite[Section 3]{JE}.

\begin{proposition} \label{prop}
Let $K$ be a semitopological semihypergroup and $f$ a continuous function on $C(K)$. Then the map $(x, y)\mapsto f(x*y)$ is separately continuous.
\end{proposition}

\begin{proposition} \label{sh1}
If $K$ is a compact semitopological semihypergroup, then 
$$WAP_r(K) = C(K).$$
\end{proposition}
\begin{proof}

Pick any $f\in C(K)$. We need to show that $\mathcal{O}_r(f)$ is weakly compact in $C(K)$.

\vspace{0.05cm}

\noi Let $\{x_\alpha\}$ be a net in $K$ converging to $x$. By Proposition \ref{prop} we see that for each $y\in K$ the map $x\mapsto f(x*y)=R_xf(y)$ is continuous and hence $R_{x_\alpha}f(y) \rightarrow R_xf(y)$. Thus the map $x\mapsto R_xf$ from $K$ into $C(K)$ is continuous where $C(K)$ is equipped with the topology of pointwise convergence.

\noi By Theorem \ref{grothm} we have that $\mathcal{O}_r(f)$ is relatively compact in $C(K)$ with respect to the weak topology.
\end{proof}

\noi It turns out that the result becomes much stronger if we have joint-continuity on the convolution product $*$ on $M(K)$, as opposed to separate continuity as in the above case. The semigroup versions for both the above proposition and the following theorem can be found in \cite{GL2}.

\begin{theorem} \label{sh2}
If $K$ is a compact semihypergroup, then 
$$AP(K)=WAP_l(K)=WAP_r(K)=C(K).$$
\end{theorem}
\begin{proof}
Pick any $f\in C(K)$. We need to show that $\mathcal{O}_r(f)$ is relative compact in $C(K)$ with respect to the strong (norm) topology.

\noi Consider the map $\Phi:K\ra C(K)$ given by $x\mapsto R_xf$. Since the map $(\mu, \nu) \mapsto \mu*\nu$ is continuous on $M^+(K)\times M^+(K)$ we have that the map $(x, y) \mapsto f(x*y)$ is continuous on $K\times K$.

\noi Fix $y_0\in K$ and pick any $\varepsilon >0$. Then for each $x\in K$ we will get open neighborhoods $V_x$ of $x$ and $W_x$ of $y_0$ such that
$$ | f(x*y_0) - f(s*t) | < \varepsilon $$
\noi where $(s, t) \in V_x \times W_x$.

\noi Since $\{V_x\}_{x\in K}$ is an open cover of the compact space $K$, we will have a finite subcover $\{V_{x_i}\}_{i=1}^n$ that covers $K$. Set $W:= \cap_{i=1}^n W_{x_i}$. Then for any $x\in K, t\in W$ 
$$ |R_{y_0}f(x) - R_tf(x)| = | f(x*y_0) - f(x*t) | < \varepsilon .$$

\noi Thus we get an open neighborhood $W$ of $y$ such that $||R_{y_0}f - R_tf ||_\infty < \varepsilon$ for any $t\in W$. Hence the map $\Phi$ is continuous and so $\mathcal{O}_r(f)$ is compact.

\noi Since $AP(K) \subset WAP_l(K)$ and $AP(K) \subset WAP_r(K)$ the result follows.
\end{proof}

\noi The immediate question that naturally rises now is whether converses to Proposition \ref{sh1} and Theorem \ref{sh2} also hold. The answer is yes for hypergroups. For semihypergroups, the converse is partially true, which needs the introduction of some more classes and techniques regarding semihypergroups. This will be discussed in details in another upcoming article by the author. Before we proceed further, let us first note the following property of (weakly) almost periodic functions on a semitopological semihypergroup.

\begin{proposition} \label{sh3}
Let $K$ be a semitopological semihypergroup and $f\in C(K)$. If $\mathcal{O}_r(f)$ is (weakly) relatively compact in $C(K)$ then the map $x\mapsto R_xf$ is (weakly) continuous on $K$.
\end{proposition}
\begin{proof}
First assume that $\mathcal{O}_r(f)$ is relatively compact in $C(K)$ in norm topology. Let $\{x_\alpha\}$ be a net in $K$ converging to $x\in K$. Then By Proposition \ref{prop} we have that $f(y*x_\alpha)\ra f(y*x)$ for each $y\in K$ and hence $R_{x_\alpha}f \ra R_xf$ pointwise in $C(K)$.

\noi If the net $\{R_{x_\alpha}f\}$ has a limit point in $C(K)$, it has to be $R_xf$. But $\overline{\mathcal{O}_r(f)}$ is compact in $C(K)$ and hence $R_{x_\alpha}f \ra R_xf$, as required.

\noi The case where $\mathcal{O}_r(f)$ is relatively compact in $C(K)$ with respect to the weak topology, can be proved proceeding along the same lines.
\end{proof}

\noi Note that all the above results will also hold true for the spaces of left almost periodic functions. The next result gives us a sufficient condition for the existence of a left-invariant mean on a semihypergroup. Recall that a (semitopological) semihypergroup $K$ is called \textit{commutative} whenever we have that $\mu * \nu = \nu * \mu$ for all $\mu, \nu \in M(K)$, or equivalently, whenever $p_x*p_y = p_y*p_x$ for all $x, y \in K$.

\begin{theorem}\label{amen1}
Let $K$ be a commutative semihypergroup. Then there exists a LIM on $C(K)$.
\end{theorem}
\begin{proof}
Consider the function $\Phi : C(K) \ra C(P_C(K))$ introduced in the proof of Theorem \ref{wapl=r}, which was defined as  $f\mapsto \tilde{f}$ where $\tilde{f}(\mu):= \int_K f \ d\mu$ for each $\mu\in P_C(K)$.

\noi We know that $(P_C(K), *)$ is a commutative semigroup. There there exists a LIM $m$ on $P_C(K)$ \cite{GL2,MI}. define $\tilde{m}:= m\circ \Phi:C(K) \ra \mathbb{C}$, \textit{i.e,} $\tilde{m}(f) = m(\tilde{f})$ for each $f\in C(K)$.

\noi Note that for any $f\in C(K)$ such that $f\geq 0$ we have that $\tilde{f}(\mu) = \int_K f \ d\mu \geq 0$ for any $\mu \in P_C(K)$. Therefore $\tilde{m}(f)=m(\tilde{f})\geq 0$. Moreover, $\tilde{m}(1) = m(\tilde{1}) = m(1)=1$. Hence $\tilde{m}$ is a mean on $C(K)$.

\noi Now pick any $x\in K$, $f\in C(K)$. For any $\mu\in P_C(K)$ we have that
\begin{eqnarray*}
(L_x f)^{\tilde{}} (\mu) &=& \int_K L_x f \ d\mu\\
&=& \int_K f \ d(p_x*\mu)\\
&=& \tilde{f}(p_x*\mu) = L_{p_x}\tilde{f} (\mu).
\end{eqnarray*}
\noi Now the left invariance of $\tilde{m}$ follows since
$$ \tilde{m}(L_xf) = m((L_xf)^{\tilde{}}) = m(L_{p_x}\tilde{f}) = m(\tilde{f}) = \tilde{m}(f)\ .$$
\end{proof}

\noi In the final part of this section, we introduce the concept of introversion on a function space and explore how the introversion operators help us in acquiring an algebraic structure on $AP(K)^*$.

\begin{definition}
Let $\mathcal{F}$ be a translation-invariant linear subspace of $C(K)$. For each $\mu \in \mathcal{F}^*$ the left introversion operator $T_\mu$ determined by $\mu$ is the map $T_\mu : \mathcal{F} \ra B(K)$ defined as
$$ T_\mu f(x) := \mu(L_x f)$$
for each $x\in K$.

\vspace{0.05cm}

\noi Similarly, the right introversion operator $U_\mu$ determined by $\mu$ is the map $U_\mu : \mathcal{F} \ra B(K)$ given by
$$U_\mu f(x) := \mu (R_x f) .$$
\end{definition}

\vspace{0.06in}

Here $\mathcal{F}$ is called left-introverted if $T_\mu f \in \mathcal{F}$ for each $\mu \in \mathcal{F}^*$, $f\in \mathcal{F}$. Similarly, $\mathcal{F}$ is called right-introverted if $U_\mu f\in\mathcal{F}$ for each $\mu \in \mathcal{F}^*$, $f\in \mathcal{F}$. We denote by $\mathcal{B}_1$ the closed unit ball of $AP(K)^*$. Before we proceed further to define an algebraic structure on $AP(K)^*$, let us first explore some basic properties of introversion operators on $AP(K)$. The following property gives us a neccessary and sufficient condition for a function to be almost periodic, in terms of left and right introversion operators. Let us first quickly recall the following version of Mazur's Theorem and another important result from \cite{MI}.

\begin{theorem}[Mazur] \label{mazth}
Let $A$ be a compact subset of a Banach space $X$. Then the closed circled convex hull of $A$, denoted as $\overline{cco}(A)$, is also compact.
\end{theorem}

\begin{theorem} 
Let $K$ be a semihypergroup and $\mathcal{F}$ be a translation-invariant conjugation-closed linear subspace of $B(K)$ containing constant functions. For any $f\in \mathcal{F}$, the set $\{T_\mu f : ||\mu||\leq 1\}$ is the closure of $cco(\mathcal{O}_r(f))$ in $B(K)$ with respect to the topology of pointwise convergence.
\end{theorem}

\noi A proof for the above version of Mazur's Theorem can be found in \cite[Appendix A]{MI}. The above theorem is proved for topological semigroups in \cite[Section 2.2]{MI}. The proof for semihypergroups follows along similar lines.

\begin{theorem} \label{intthm0}
Let $K$ be a semitopological semihypergroup. Then $f\in AP(K)$ if and only if the map $\mu \mapsto T_\mu f : \mathcal{B}_1 \ra \mathcal{B}(K)$ is $\sigma(\mathcal{B}_1, AP(K))$-norm continuous.
\end{theorem}
\begin{proof}
Let $f\in AP(K)$. Let $\Psi: \mathcal{B}_1 \ra B(K)$ be the map given by
$$\Psi(\mu) := T_\mu f \mbox{ for each } \mu\in  \mathcal{B}_1.$$
Let $\mu_\alpha \ra \mu$ in $ \mathcal{B}_1$ with respect to the topology $\sigma(\mathcal{B}_1, AP(K))$, \textit{i.e,} we have that $\mu_\alpha(g) \ra \mu(g)$ for each $g\in AP(K)$.

\noi Also by Proposition \ref{aptr} we know that $L_x f\in AP(K)$ for each $x\in K$. Hence in particular for each $x\in K$ we have that 
\begin{eqnarray*}
\mu_\alpha(L_x f) &\ra & \mu(L_xf) \mbox{\ \ \ \ for each } x\in K \\
\Rightarrow \hspace{0.02in} T_{\mu_\alpha}f(x) &\ra &  T_\mu f(x) \mbox{ \ \ \ \ for each } x\in K
\end{eqnarray*}
\noi Hence the map $\Psi$ is continuous when $B(K)$ is equipped with the topology of pointwise convergence.

\noi Also, $\Psi(\mathcal{B}_1)$ is the closure of $cco(\mathcal{O}_r(f))$ with respect to topology of pointwise convergence. Since $\mathcal{O}_r(f)$ is compact in $B(K)$, we have that $\overline{\Psi(\mathcal{B}_1)}$ is compact. Hence the topology of pointwise convergence coincides with the norm topology on $\Psi(\mathcal{B}_1)$.

\vspace{0.05cm}

\noi Conversely, suppose the map $\Psi:  \mathcal{B}_1 \ra B(K)$ defined above is $\sigma(\mathcal{B}_1, AP(K))$-norm continuous. Then it follows immediately that $f\in AP_r(K)= AP(K)$ since $\Psi(\mathcal{B}_1)$ is norm-compact and the following inclusion holds:
$$ \mathcal{O}_r(f) \subset \overline{cco(\mathcal{O}_r(f))}^{{ pt.wise \ topology}} = \Psi(\mathcal{B}_1) .$$
\end{proof}

\noi Of course, the right-counterpart of the above theorem holds true in a similar manner, \textit{i.e,} we can also show that a function $f$ is almost periodic if and only if the map $\mu \mapsto U_\mu f : \mathcal{B}_1 \ra \mathcal{B}(K)$ is $\sigma(\mathcal{B}_1, AP(K))$-norm continuous.

\vspace{0.05cm}

\noi Now we see that $AP(K)$ is left and right introverted, enabling us to introduce left and right Arens product \cite{AR} on $AP(K)^*$.

\begin{corollary}\label{introcor}
Let $K$ be any semitopological semihypergroup. Then $AP(K)$ is left and right introverted.
\end{corollary}
\begin{proof}
We know that $AP(K)$ is a translation-invariant linear subspace of $C(K)$. Now pick any $f\in AP(K)$ and again consider the function $\Psi: \mathcal{B}_1 \ra B(K)$ given by
$$\Psi(\mu) := T_\mu f \mbox{ for each } \mu\in  \mathcal{B}_1.$$ 
As pointed out in the preceding proof, we have that the norm topology coincides with the topology of pointwise convergence on $\Psi(\mathcal{B}_1)$ and hence finally we have that
$$ \Psi(\mathcal{B}_1) = \overline{cco} (\mathcal{O}_r(f)) \subset AP(K) .$$
Hence after scaling by proper scalars we have that $T_\mu f \in AP(K)$ for each $\mu \in AP(K)^*$, as required.

\noi The proof for the right-counterpart follows similarly.
\end{proof}

\begin{theorem} \label{intro00}
Let $K$ be a semitopological semihypergroup and $f\in AP(K)$.  Then the map $\Phi: \mathcal{B}_1 \times \mathcal{B}_1 \ra \mathbb{C}$ given by
$$ \Phi(\mu, \nu) := \mu(T_\nu f)$$
is continuous with respect to the topology $\sigma(\mathcal{B}_1, AP(K)) \times \sigma(\mathcal{B}_1, AP(K))$.

\vspace{0.05cm}

\noi Moreover, for any $\mu, \nu \in AP(K)^*$ we have that
$$\mu(T_\nu f) = \nu (U_\mu f).$$
\end{theorem}
\begin{proof}
\noi Pick any $(\mu_0, \nu_0)\in \mathcal{B}_1 \times \mathcal{B}_1$ and consider the function $\Psi$ as in Theorem \ref{intthm0}. Then for any $\mu, \nu \in \mathcal{B}_1$ we have that
\begin{eqnarray*}
|\Phi(\mu, \nu) - \Phi(\mu_0, \nu_0)| &=& |\mu(T_\nu f) - \mu_0(T_{\nu_0}f)|\\
&\leq & ||\Psi(\nu) - \Psi(\nu_0)|| + |(\mu-\mu_0) (\Psi(\nu_0))|
\end{eqnarray*}
\noi Since $\Psi$ is $\sigma(\mathcal{B}_1, AP(K))$-norm continuous, we have that both the terms in the last inequality tends to zero whenever $(\mu, \nu)$ tends to $(\mu_0, \nu_0)$ in $\mathcal{B}_1 \times \mathcal{B}_1$ with respect to the topology $\sigma(\mathcal{B}_1, AP(K)) \times \sigma(\mathcal{B}_1, AP(K))$.

\vspace{0.03in}

\noi In a similar way, we can show that the map $\Phi': \mathcal{B}_1 \times \mathcal{B}_1 \ra \mathbb{C}$ given by
$$ \Phi'(\mu, \nu) := \nu(U_\mu f)$$
is continuous with respect to the topology $\sigma(\mathcal{B}_1, AP(K)) \times \sigma(\mathcal{B}_1, AP(K))$.

\noi Now for any $x\in K$ consider the evaluation map $E_x \in AP(K)^*$ given by 
$$ E_x (f) := f(x) \mbox{ \ \ for each } f\in AP(K) .$$
We denote the set of all evaluation maps in $AP(K)^*$ as $\mathcal{E}(K)$,\textit{ i.e,} we have 
$$\mathcal{E}(K) := \{E_x : x\in K \}. $$
\noi Pick any $\mu= E_x, \nu = E_y$ for some $x, y \in K$. Then we have
\begin{eqnarray*}
\mu(T_\nu f)&=& E_y (L_x f)\\
&=& f(x*y)\\
&=& E_x(R_y f)\\
&=& E_y(U_\mu f) = \nu(U_\mu f).
\end{eqnarray*}
Since $\mathcal{E}(K)$ serves as the set of extreme points, we have that the weak$^*$ closure of $cco(\mathcal{E}(K))$ equals $\mathcal{B}_1$ \cite{MI}. Hence it follows from the continuity of the maps $\Phi$ and $\Phi'$ that 
$$\mu(T_\nu f) = \nu (U_\mu f) $$
for any $\mu, \nu \in \mathcal{B}_1$ and hence by proper scaling, for any $\mu, \nu \in AP(K)^*$.
\end{proof}

\noi Note that the above theorem holds true even if we replace $AP(K)$ by any translation invariant conjugation-closed linear subspace $\mathcal{F}$ of $C(K)$ containing constant functions.

\noi Now let us define a product $\star$ on $AP(K)^*$ given by
$$\mu \star \nu (f) := \mu(T_\nu f).$$
for any $(\mu, \nu) \in AP(K)^* \times AP(K)^*$.

\noi Now recall the construction of the left Arens product $\lozenge$ for any Banach algebra $X$ \cite{AR} and apply it on $AP(K)^*$. For any $\mu, \nu \in AP(K)^*$, $f\in AP(K)$ and $x, y \in K$ we have that
\begin{eqnarray*}
f\lozenge x (y) &:=& f(x*y).\\
\nu \lozenge f (x) &:=& \nu(f\lozenge x).\\
\mu \lozenge \nu (f) &:=& \mu (\nu \lozenge f).
\end{eqnarray*}

\noi Hence in particular, we have that
\begin{eqnarray*}
f\lozenge x (y) &=& L_xf(y).\\
\nu \lozenge f (x) &=& \nu(L_xf) \ = \ T_\nu f(x).\\
\mu \lozenge \nu (f) &=& \mu(T_\nu f) \ = \ \mu \star \nu (f).
\end{eqnarray*}
\noi Note that here the second step is possible since $AP(K)$ is left translation-invariant and the last step is possible since $AP(K)$ is left introverted.

\noi Thus the left Arens product coincides with $\star$ on $AP(K)^*$, and hence $(AP(K)^*, \star)$ becomes a Banach algebra.

\noi Now in a similar manner, consider the right Arens product $\square$ on $AP(K)^*$. Then we have that
\begin{eqnarray*}
x\square f (y) &=& f(y*x) \ = \ R_xf(y).\\
f\square \mu (x) &=& \mu(x\square f) \ = \ \mu(R_xf) \ = \ U_\mu f(x).\\
\mu \square \nu (f) &=& \nu(f\square \mu) \ = \ \nu(U_\mu f).
\end{eqnarray*}
Thus in view of Theorem \ref{intro00} we see that the left and right Arens products coincide on $AP(K)^*$, \textit{i.e,} $AP(K)^*$ is Arens regular.

\section{Homomorphisms and Ideals}

\noi In this section, we first introduce a more general form of homomorphism in (semitopological) semihypergroups. Later we introduce the concept of an ideal in a semihypergroup and obtain some basic properties and relations between (closed) ideals and homomorphisms. Finally, we investigate the structure of the kernel of a compact (semitopological) semihypergroup and the connection between minimal left ideals and amenability in compact semihypergroups. Unless otherwise mentioned, $K$ and $H$ will denote semitopological semihypergroups.

\begin{definition} \label{homodef}
A continuous map $\phi:K\ra H$ is called a homomorphism if for any Borel measurable function $f$ on $H$ and for any $x, y \in K$ we have that
$$ f\circ \phi(x*y) = f(\phi(x)*\phi(y)) .$$

\vspace{0.05cm}

\noi A homomorphism $\phi$ between $K$ and $H$ is called an isomorphism if the map $\phi$ is bijective.
\end{definition}

\begin{remark} 
\noi In $1975$ Jewett introduced the concept of orbital morphisms for hypergroups in \cite{JE}. If $K$ and $J$ are two hypergroups, then a map $\phi: K \ra J$ is called proper if $\phi^{-1}(C)$ is compact in $K$ for every compact set $C$ in $J$. A recomposition of $\phi$ was defined to be a continuous map from $J \ra M^+(K)$ defined as $x\mapsto q_x$ such that $supp (q_x) = \phi^{-1}(\{x\})$.

\noi Roughly speaking, a proper open map $\phi:K \ra J$ is called a orbital homomorphism if for any $x, y \in J$ we have that $p_x*p_y = \phi(q_x*q_y)$ as measures, \textit{i.e,} for any Borel measurable function $f$ on $J$ we have that $\int_J f \ d(p_x *p_y) = \int_K f\circ \phi \ d(q_x *q_y)$.

\noi Jewett defined homomorphism between hypergroups as a special case of orbital morphisms. According to the definition in \cite{JE} a continuous open map $\phi:K \ra J$ is called a homomorphism if $\phi(p_x*p_y) = p_{\phi(x)} * p_{\phi(y)}$ for any $x, y \in K$.

\noi Note that our definition of homomorphism is equivalent to that of Jewett's apart from the fact that in Definition \ref{homodef} we do not need the map $\phi$ to be open and proper, and there the map was defined between two (semitopological) semihypergroups, instead of hypergroups.

\noi Otherwise these two definitions coincide since similar to the definition of orbital morphisms, Jewett's definition of homomorphism implies that for any measurable function $f$ on $J$, $x, y \in K$ we have that $\int_J f \ d(p_{\phi(x)}*p_{\phi(y)}) = \int_K f\circ \phi \ d(p_x * p_y)$. Hence by definition $f(\phi(x) * \phi(y)) = f\circ \phi(x*y)$.
\end{remark}

\noi Now we check whether some of the basic properties of the classical case hold for this definition as well. Since in the classical case the multiplication of two points simply gives us another point, these properties hold trivially in that case. 

\begin{lemma} \label{lemma1}
Let $K$ and $H$ be (semitopological) semihypergroups, and $\phi: K \ra H$ a homomorphism. Then for any $x, y \in K$, for almost all $z\in supp(p_x *p_y)$ with respect to the measure $(p_x*p_y)$ we have  
$$\phi(z) \in supp (p_{\phi(x)}* p_{\phi(y)}).$$ 
Conversely, for any $x,y,z\in K$, for almost all $\phi(z) \in supp (p_{\phi(x)} *p_{\phi(y)})$ with respect to the measure $(p_{\phi(x)}*p_{\phi(y)})$ we have 
$$z\in \phi^{-1}\phi( supp (p_x*p_y)).$$
\end{lemma}
\begin{proof}
Pick any $x,y \in K$ and let $z\in supp(p_x*p_y)$. Now set $A:= supp(p_{\phi(x)} *p_{\phi(y)}) \subset H$ and define a measurable function $g$ on $H$ such that $g\equiv 0$ on $A$, $g\equiv1$ on $A^c$. Then we have that
\begin{eqnarray*}
0 = \int_A g \ d(p_{\phi(x)} *p_{\phi(y)}) &=& \int_H g \ d(p_{\phi(x)} *p_{\phi(y)})\\
&=& g(\phi(x)*\phi(y))\\
&=& g\circ\phi(x*y)= \int_K g\circ\phi \ d(p_x*p_y).
\end{eqnarray*}
\noi For any $z\in {supp}(p_x*p_y)$, for the relation $\int_K (g\circ\phi) \ d(p_x*p_y) =0$ to hold true we must have that $g\circ\phi (z) =0$ almost everywhere on $supp(p_x*p_y)$ with respect to $(p_x*p_y)$. Thus by construction of $g$ we have that $\phi(z) \in A$ for almost all $z\in supp(p_x*p_y)$ as required.

\noi Now to prove the converse, pick any $x, y \in K$ and let $z\in K$ be such that $\phi(z)\in supp(p_{\phi(x)} *p_{\phi(y)})$.

\noi Set $A:= supp(p_x*p_y)$ and define a measurable function $h$ on $H$ such that $h \equiv 0$ on $\phi(A)$ and $h\equiv 1$ on $\phi(A)^c$. Note that for any $u\in A$ we have that $\phi(u)\in \phi(A)$ and hence $h\circ\phi(u) =0$. Thus we have
\begin{eqnarray*}
0 = \int_A h\circ \phi(u) \ d(p_x*p_y) &=& \int_K h\circ\phi \ d(p_x*p_y)\\
&=&h\circ\phi(x*y)\\
&=& h(\phi(x)*\phi(y))= \int_H h \ d(p_{\phi(x)} *p_{\phi(y)}).
\end{eqnarray*}
\noi where the first equality follows from the construction of $h$ and the second equality follows from the construction of $A$.

\noi Now if $z\in A$, then the result follows trivially. Suppose $z$ does not lie in $A$. Since $\phi(z) \in supp(p_{\phi(x)} *p_{\phi(y)}) \subset H$, for the relation $\int_H h \ d(p_{\phi(x)} *p_{\phi(y)}) =0$ to be true we must have that $h(\phi(z))=0$ for almost all $\phi(z)\in supp(p_{\phi(x)} *p_{\phi(y)})$ with respect to $(p_{\phi(x)} *p_{\phi(y)})$. Thus almost all $\phi(z) \in \phi(A)$ and so $z\in \phi^{-1}\phi(A)$ as required.
\end{proof}

\begin{remark}
Roughly speaking, the above lemma implies that for an isomorphism $\phi:K \ra H$ we have the following:

\noi For any $x, y \in K$ we have that $z\in supp (p_x*p_y)$ if and only if $\phi(z) \in supp (p_{\phi(x)}*p_{\phi(y)})$, up to a set of measure zero.
\end{remark}

\begin{proposition}
Let $\phi:K \ra H$ be a homomorphism. Then $\phi(K)$ is a semihypergroup.
\end{proposition}
\begin{proof}
In order to verify this, we only need to check that $\phi(K)$ is closed under convolution, \textit{i.e,} $\phi(K)*\phi(K) \subset \phi(K)$.

\noi Pick $a, b \in K$ and set $A:= supp(p_{\phi(a)} * p_{\phi(b)})$. If possible, suppose $A\nsubseteq \phi(K)$ so that  $A \setminus \phi(K) \neq \O$. Now define a measurable function $f$ on $H$ such that $f \equiv 1$ on $A\setminus \phi(K)$ and $f \equiv 0$ on $A \cap \phi(K)$. Then we have
\begin{eqnarray*}
\int_{A\setminus \phi(K)} f \ d(p_{\phi(a)} * p_{\phi(b)}) = \int_A f \ d(p_{\phi(a)} * p_{\phi(b)}) &=& \int_H f \ d(p_{\phi(a)} * p_{\phi(b)})\\
&=& f(\phi(a)*\phi(b))\\
&=& f\circ\phi(a*b) = \int_K f\circ\phi \ d(p_a*p_b) = 0.
\end{eqnarray*}
\noi where the first equality holds since $f\equiv 0$ on $A\cap \phi(K)$, the second equality follows from the construction of $A$ and the last equality holds since for almost all $z\in supp(p_a*p_b)$ we have from Lemma \ref{lemma1} that $\phi(z) \in A$. Thus $\phi(z) \in A \cap \phi(K)$ and hence $f\circ \phi(z) = 0$.

\noi But $f\equiv 1$ on $A\setminus \phi(K)$ and $A\setminus\phi(K) \subset supp(p_{\phi(a)} * p_{\phi(b)})$. So $A\setminus \phi(K)$ can not be a set of $(p_{\phi(a)} * p_{\phi(b)})$-measure zero and hence $\int_{A\setminus \phi(K)} f \ d(p_{\phi(a)} * p_{\phi(b)}) \neq 0$, and we arrive at a contradiction.
\end{proof}
We now introduce and study ideals in (semitopological) semihypergroups.
\begin{definition}
Let $K$ be a (semitopological) semihypergroup. A subset $I\subset K$ is called a left ideal of $K$ if for any $a\in I$, $x\in K$ we have that $p_x*p_a (I) =1$.

\vspace{0.03cm}

\noi Equivalently, a subset $I\subset K$ is called a left ideal of $K$ if for any $a\in I$, $x\in K$ and any Borel measurable set $E\subset K$ such that $E\cap I = \O$, we have that $p_x*p_a(E) = 0$.
\end{definition}

\noi Note that a closed subset $I\subset K$ is a left ideal in $K$ if and only if for any $x\in K$, $a\in I$ we have that $supp(p_x*p_a)\subset I$. Similar definitions and comments apply to right ideals in $K$. Throughout this section, we discuss the properties of left ideals and their relations to several algebraic and analytic objects associated to $K$. All these results have similar counterparts for right ideals too. Henceforth all ideals will be assumed to be closed, unless otherwise specified.

\begin{proposition} \label{iex}
Let $K$ be a semitopological semihypergroup. Pick any $a\in K$. Then $I:=K*\{a\}$ is a left ideal in $K$.
\end{proposition}
\begin{proof}
Pick any $x\in K$ and $b\in I$. Since $I = \cup_{y\in K} supp(p_y*p_a)$ there exists some $y_0\in K$ such that $b\in supp(p_{y_0}*p_a)= \{y_0\}*\{a\}$. Now the result follows as
\begin{eqnarray*}
supp(p_x*p_b) = \{x\} * \{b\} &\subset& \{x\} * (\{y_0\}*\{a\})\\
&=& (\{x\}*\{y_0\}) * \{a\} \subset K*\{a\} = I.
\end{eqnarray*}
\end{proof}

\begin{proposition}
Let $\phi:K \ra H$ be a homomorphism  and $I\subset K$ a left ideal. Then $\phi(I)$ is also a left ideal in $\phi(K)$.
\end{proposition}
\begin{proof}
Pick any $a\in I$, $x\in K$. Set $B:= supp(p_{\phi(x)}*p_{\phi(a)})$ and define a function $f$ on $H$ such that $f\equiv 1$ on $\phi(I)\cap B$ and $f\equiv 0$ elsewhere. Then we have
\begin{eqnarray*}
p_{\phi(x)}*p_{\phi(a)}(\phi(I)) = \int_{\phi(I)\cap B} 1 \ d(p_{\phi(x)}*p_{\phi(a)})&=& \int_H f \ d(p_{\phi(x)}*p_{\phi(a)})\\
&=& f(\phi(x) * \phi(a))\\
&=& f\circ \phi (x*a)\\
&=& \int_K (f\circ \phi) \ d(p_x*p_a)\\
&=& \int_{\phi^{-1}\phi(I)}1 \ d(p_x*p_a)\\
&=& (p_x*p_a)(\phi^{-1}\phi(I))= (p_x*p_a) (I) = 1.
\end{eqnarray*}
\noi where the sixth equality follows from the fact that for almost all $z\in supp(p_x*p_a)$ Lemma \ref{lemma1} gives us that $\phi(z)\in B$ and hence $f(\phi(z)) = 0$ whenever $\phi(z)$ does not lie in $\phi(I)$, \textit{i.e,} whenever $z$ lies outside $\phi^{-1}\phi(I)$. Also, the second last equality follows since $I\subset \phi^{-1}\phi(I)$ and $supp(p_x*p_a) \subset I$ as $I$ is a left ideal in $K$.
\end{proof}

\begin{proposition}
Let $\phi: K \ra H$ be a homomorphism and $J\subset H$ a left ideal in $H$. Then $\phi^{-1}(J)$ is also a left ideal in $K$.
\end{proposition}
\begin{proof}
Pick any $a\in \phi^{-1}(J)$, $x\in K$. Then $\phi(a)\in J$ and hence $supp(p_{\phi(x)}*p_{\phi(a)})\subset J$ and $(p_{\phi(x)}*p_{\phi(a)}) (J) =1$. Define a function $f$ on $H$ by $f\equiv 1$ on $J$ and $f\equiv 0$ elsewhere.

\noi Now for almost all $z\in supp(p_x*p_a)$ using Lemma \ref{lemma1} we have that 
$$\phi(z)\in supp(p_{\phi(x)}*p_{\phi(a)}) \subset J.$$
Hence $f\circ \phi(z) =1$ for almost all $z\in supp(p_x*p_a)$. Hence the result follows as
\begin{eqnarray*}
(p_x*p_a) (\phi^{-1}(J)) = \int_{\phi^{-1}(J)} 1 \ d(p_x*p_a) &=& \int_K f\circ \phi \ d(p_x*p_a)\\
&=& f\circ\phi(x*a)\\
&=&f(\phi(x)*\phi(a))\\
&=& \int_H f \ d(p_{\phi(x)}*p_{\phi(a)})\\
&=& \int_J 1 \ d(p_{\phi(x)}*p_{\phi(a)}) = (p_{\phi(x)}*p_{\phi(a)}) (J) = 1.
\end{eqnarray*}
\end{proof}

\noi A left (resp. right) ideal $I\subset K$ is called a minimal left (resp. right) ideal of $K$ if $I$ does not contain any proper left (resp. right) ideal of $K$. An ideal $I$ of $K$ which is both a minimal left and right ideal, is called a minimal ideal of $K$. We now start off with examining some equivalence criteria for the minimality of a left ideal. Next we examine some basic properties of minimal left (resp. right) ideals that hold trivially for semigroups, for reasons explained in the previous section.

\begin{proposition} \label{iprop0}
For any left ideal $I\subset K$ the following are equivalent:
\begin{enumerate}
\item $I$ is a minimal left ideal.
\item $K*\{a\} = I$ for any $a\in I$.
\item $I*\{a\} = I$ for any $a\in I$.
\end{enumerate}
\end{proposition}
\begin{proof}
\noi $(1) \Rightarrow (3)$: Let $I$ be a minimal left ideal of $K$ and $a\in I$. Then $I*\{a\} \subset I*I\subset I$ since $I$ is a left ideal. Also, $I*\{a\}$ is a left ideal since $K*(I*\{a\}) = (K*I)*\{a\} \subset I*\{a\}$.

\noi Hence the minimality of $I$ gives us that $I*\{a\} =I$.

\vspace{0.05cm}

\noi $(3) \Rightarrow (2)$: Since $I$ is a left ideal, for each $a\in I$ we have $I=I*\{a\} \subset K*\{a\} \subset I$, which forces that $K*\{a\} = I$.

\vspace{0.05cm}

\noi $(2) \Rightarrow (1)$: Let $J$ be a left ideal contained in $I$. Pick any $b\in J$. Since $b\in I$ as well, (2) gives us $I= K*\{b\} \subset K*J \subset J$. Therefore $I=J$ implying the minimality of $I$ in $K$.
\end{proof}

\begin{corollary} \label{iprop1}
For any (semitopological) semihypergroup $K$ the following assertions hold.
\begin{enumerate}
\item If $I_1, I_2$ are minimal left ideals in $K$, then either $I_1=I_2$ or $I_1\cap I_2 = \O$.
\item Let $I$ be a minimal left ideal in $K$. Then any minimal left ideal in $K$ will be of the form $I*\{x\}$ for some $x\in K$. Moreover, we have that $K*J =J$ for any minimal left ideal $J$ in $K$.
\item $K$ can have at most one minimal ideal, namely the intersection of all  ideals in $K$.
\end{enumerate}
\end{corollary}
\begin{proof}
The proofs of $(1)$ and $(3)$ are straight-forward. To prove $(2)$: let $J$ be any minimal left ideal of $K$. Pick $y\in J$. Then $I*\{y\} \subset I*J\subset J$ and the minimality of $J$ forces that $J=I*\{y\}$. Now pick any $x_0\in K$ and consider the left ideal $I_0 := I*\{x_0\}$. It follows readily from Proposition \ref{iprop0} that $K*I = \cup_{a\in I} K*\{a\} = \cup_{a\in I} I =I$. Therefore $K*I_0 = K*(I*\{x_0\}) = (K*I)*\{x_0\} = I*\{x_0\}=I_0$.
\end{proof}

\begin{remark} \label{idealdevi}
Unlike in the semigroup setting \cite{GL1, GL2}, the family of sets $\{I*\{x\} : x\in K\}$ does not serve as an exhaustive family of minimal left ideals for a semihypergroup, where $I$ is a minimal left ideal in $K$. For any $x\in K$ the set $I*\{x\}$ is of course a left ideal of $K$, but it need not be minimal. For example, simply consider the finite semihypergroup $(K,*)$ where $K=\{a, b, c\}$ and the operation $*$ given in the following table. 
\begin{center}
  \begin{tabular}{ c | c  c  c } \label{ext}
    $*$ & $p_a$ & $p_b$ & $p_c$ \\
    \hline
    $p_a$ & $p_a$ & $\frac{1}{2}(p_a+p_b)$ & $\frac{1}{2}(p_a+p_c)$ \\ 
    $p_b$ & $p_a$ & $\frac{1}{2}(p_a+p_b)$ & $\frac{1}{2}(p_a+p_c)$ \\
    $p_c$ & $p_a$ & $p_b$ & $p_c$ \\
  \end{tabular}
  \end{center}

\noi Set $I =\{a\}$. Then $I$ is a minimal left ideal as the left multiplication by $*$ leaves $p_a$ fixed. But $I*\{b\}=\{a, b\}$ and $I*\{c\} = \{a, c\}$ are both left ideals that fail to be minimal.
\end{remark}

\begin{proposition} \label{ithm0}
Let $K$ be a compact semitopological semihypergroup. Then each left ideal of $K$ contains at least one minimal left ideal and each right ideal contains at least one minimal right ideal of $K$.

\vspace{0.05cm}

\noi Moreover, each minimal left and right ideal of $K$ is closed.
\end{proposition}
\begin{proof}
We will prove the statement only for minimal left ideals.

\noi Let $I$ be a left ideal of $K$. Consider the following collection of left ideals in $K$ 
$$ \mathcal{Q} := \{J \subset K : J \mbox{ is a left ideal in } K \mbox{ and } J \subset I\}.$$
If $a \in I$, then $K*\{a\}\subset K* I \subset I$ is a left ideal in $K$, by Proposition \ref{iex}. Hence $\mathcal{Q}$ is non-empty and non-trivial. Equip $\mathcal{Q}$ with the partial order of reverse inclusion. Let $\mathcal{C}$ be a linearly ordered sub-collection of $\mathcal{Q}$. Since $K$ is compact, the ideal $\cap_{J\in \mathcal{C}} \ J$ is non-empty. Hence we can use Zorn's Lemma to deduce that there exists a minimal element $J_0$ in $\mathcal{Q}$, which serves as a minimal left ideal of $K$ contained in $I$.

\noi Now let $a\in J_0$. Then $K*\{a\}$ is a left ideal in $K$ contained in $J_0$. Hence by minimality of $J_0$ we get that $J_0 = K*\{a\}$, which is closed as $K$ is compact \cite{JE}.
\end{proof}

\noi Finally in the last part of this section, we define and explore the characteristics of the kernel of a (semitopological) semihypergroup. We restrict ourselves to the case where the underlying space is compact and investigate the interplay between the structures of a kernel and the exhaustive set of minimal left (resp. right) ideals. We conclude with a result outlining the relationship between amenability and the existence of a unique minimal left ideal for a compact semihypergroup.

\vspace{0.03in}

\noi The kernel of a (semitopological) semihypergroup $K$ denoted as $Ker(K)$ is defined to be the intersection of all (two-sided) ideals in $K$, \textit{i.e,} we define
\begin{equation*}
 Ker(K) := \bigcap_{\substack{I \subset K  \\
                     \mbox{ is an ideal }}}
                      I .
\end{equation*}
The set of all minimal left and right ideals of $K$ are denoted as $\mathcal{L}(K)$ and $\mathcal{R}(K)$ respectively. Finally, the following theorems give us an explicit idea of the structure of the kernel of a compact semihypergroup.

\begin{theorem} \label{ithm1}
Let $K$ be a compact semitopological semihypergroup and $I$ be a minimal left ideal in $K$. Then we have that
$$\bigcup_{M\in \mathcal{L}(K)} \hspace{-0.07in} M \subset Ker(K) \subset \bigcup_{x\in K}  (I * \{x\}) $$
\end{theorem}
\begin{proof}
Let $I\in \mathcal{L}(K)$ and consider the following union
$$I_0 := \cup_{x\in K} I * \{x\} .$$
We know from Corollary \ref{iprop1} that $I_0$ is a union of left ideals and hence is a left ideal of $K$ itself. Now pick any $a\in I_0$. There exists some $x_0\in K$ such that $a\in I*\{x_0\}$. Thus for any $y\in K$ we have
\begin{eqnarray*}
supp(p_a*p_y) = \{a\}*\{y\} &\subset & (I*\{x_0\})*\{y\} \\
& = & I* (\{x_0\}*\{y\})\\
& = & \cup_{x\in \{x_0\}*\{y\}} \ (I*\{x\}) \subset \cup_{x\in K} \ (I*\{x\}) = I_0.
\end{eqnarray*}
\noi Hence $I_0$ is a right ideal of $K$ as well and therefore $Ker(K) \subset I_0$.

\noi Now let $J$ be any ideal in $K$ and $M\in \mathcal{L}(K)$. Then 
$$K*(J*M) = (K*J)*M \subset J*M \subset M .$$
Thus $J*M$ is a left ideal of $K$ contained in $M$ and hence by minimality of $M$, we have that $M=J*M$. But $J$ is a two-sided ideal and hence $M= J*M \subset J$.

\noi This is true for any $M\in \mathcal{L}(K)$ and any ideal $J$ in $K$, and hence finally we see that $\cup_{I\in \mathcal{L}(K)} \ I \subset Ker(K)$ as required.
\end{proof}

\noi The result holds true similarly for minimal right ideals as well, \textit{i.e,} for any minimal right ideal $J\in \mathcal{R}(K)$ we have that 
$$\bigcup_{N\in \mathcal{R}(K)} \hspace{-0.07in} N \subset Ker(K) \subset \bigcup_{x\in K}  (\{x\}*J) $$

\begin{corollary}\label{icor1}
Let $K$ be a compact semitopological semihypergroup. Then $Ker(K)$ is non-empty.
\end{corollary}
\begin{proof}
We know by Proposition \ref{iex} that $K*\{a\}$ is a left ideal for any $a\in K$. Since $K$ is compact, it follows from Proposition \ref{ithm0} that it contains at least one minimal left ideal. Hence the result follows immediately from the above theorem.
\end{proof}

\noi Note that the above Corollary only implies that $K$ does not contain any two disjoint ideals. But it may very well be the case that $Ker(K) = K$ as demonstrated in Remark \ref{ext} and Example \ref{ex2}.

\begin{theorem} \label{ithm3}
Let $K$ be a compact semihypergroup. If $K$ is right amenable, \textit{i.e,} if $C(K)$ admits a right invariant mean, then $K$ has a unique minimal left ideal.
\end{theorem}
\begin{proof}
\noi Let $m$ be a right invariant mean on $C(K)$. Since $K$ is compact, it follows from Proposition \ref{ithm0} that there exists at least one minimal left ideal of $K$.

\noi If possible, let $I_1$ and $I_2$ be two distinct minimal left ideals of $K$. Then by Corollary \ref{iprop1} and Proposition \ref{ithm0} we have that $I_1\cap I_2 = \varnothing$ and both $I_1$ and $I_2$ are closed in $K$. Since $K$ is compact, it is normal. Hence we can use Urysohn's Lemma to get a continuous function $f\in C(K)$ such that $f\equiv 0$ on $I_1$ and $f\equiv 1$ on $I_2$.

\noi Now pick $a\in I_1$. For any $x\in K$ we have that
$$R_af(x) = f(x*a) = \int_{supp(p_x*p_a)} f \ d(p_x*p_a) = 0$$
where the last equality follows since $supp(p_x*p_a)\subset I_1$ and hence $f\equiv 0$ on $supp(p_x*p_a)$.

\noi Similarly for any $b\in I_2$, $x\in K$ we have that
$$R_bf(x) = f(x*b) = \int_{supp(p_x*p_b)} f \ d(p_x*p_b) = (p_x*p_b)(K) =1$$
where as before, the fourth equality follows since $supp(p_x*p_b)\subset I_2$ and hence $f\equiv 1$ on $supp(p_x*p_b)$.

\noi Thus for any $a\in I_1$, $b\in I_2$ we have that $R_af \equiv 0$ and $R_bf \equiv 1$. This leads to a contradiction since 
$$ 1 = m(1) = m(R_bf) = m(f) = m(R_af) = m(0) .$$
Hence the minimal left ideal of $K$ must be unique.
\end{proof}


\section{Open Questions}

\noi As discussed briefly in the introduction, the lack of prior extensive research since its inception and the significant examples available in coset and orbit spaces of locally compact groups, Lie groups and homogeneous spaces, makes way to a number of exciting new avenues of research on semihypergroups. 

\noi The author has already shown (in an upcoming article) that a concept of free products can naturally be introduced to semihypergroups, giving rise to a new class of examples. What makes the structure more intriguing is the fact that the natural algebraic free-product structure, which does not work for topological groups, indeed works in this case for a vast class of semihypergroups, including nontrivial coset and orbit spaces.

\noi  Here we list some of the immediate potential problems and areas of semihypergroups that we are currently working on and/or intend to work on in the near future.

\vspace{0.03 in}

\noi \textbf{Problem 1:}  Use the algebraic structure imposed on $AP(K)^*$ to acquire a general compactification of semihypergroups.

\vspace{0.03in}

\noi \textbf{Problem 2:} Investigate the set of minimal ideals on a (semitopological) semihypergroup more closely, finally to explore its relation to the space of almost periodic and weakly almost periodic functions.

\vspace{0.03in}

\noi \textbf{Problem 3:} Investigate the idempotents of a compact semihypergroup and explore their relation to the space of minimal left ideals, kernel and amenability of a semihypergroup.

\vspace{0.03in}

\noi \textbf{Problem 4:} Explore if results equivalent to isomorphism theorems hold true for a semihypergroup, in addition to exploring the structure of the kernel of a homomorphism for a (semitopological) semihypergroup with identity, as in Example \ref{ex3} and Example \ref{ex4}.

\vspace{0.03in}

\noi \textbf{Problem 5:} Study different notions of amenability on (semitopological) semihypergroups, specially for the non-commutative case.

\vspace{0.1in}

\noi \textbf{Problem 6:} Investigate the F-algebraic structure (defined by A. T. Lau in \cite{LA2}) on the measure algebra of a semihypergroup.


\section{Acknowledgement}

\noi The author would like to thank her doctoral thesis advisor  Dr. Anthony To-Ming Lau for suggesting the topic of this paper and for the helpful discussions during the course of this work. She would also like to thank the referee for the valuable suggestions that led to a better presentation of the article.


\end{document}